\documentclass{article}

\usepackage[english]{babel}
\usepackage{amsmath}
\usepackage{amsthm}
\usepackage{amssymb}
\usepackage{braket}
\usepackage{enumitem}
\usepackage{etoolbox}
\usepackage{graphicx}
\usepackage{hyperref}
\usepackage[capitalize]{cleveref}
\usepackage{ifthen}
\usepackage{interval}
\usepackage{mathtools}
\usepackage{mleftright}
\usepackage{thmtools}
\usepackage{thm-restate}
\usepackage{xcolor}
\usepackage{xargs}

\title{A connection between the Mahler conjecture and floating bodies}
\author{Agam Guberman\thanks{Tel Aviv University, Israel. Supported by the Israel Science Foundation (grant 1626/25).}}
\declaretheorem[within=section]{definition}
\declaretheorem[numberlike=definition]{conjecture, fact, theorem, lemma, proposition, remark, corollary, example}



\declaretheorem[sibling=claim, numbered=no, name=Claim]{claim*}
\AtBeginEnvironment{proof}{\setcounter{claim}{0}}
\newenvironment{claimproof}[1][]
{\par\noindent\ifthenelse{\equal{#1}{}}{\textit{Proof.}\space}{\textit{Proof of #1.}}}
{\hfill$\square$\newline}

\intervalconfig{soft open fences}

\newcommand{\parens}[1]{\mleft(#1\mright)}

\newcommand{\curlies}[1]{\mleft\{#1\mright\}}
\newcommand{\triangles}[1]{\mleft\langle#1\mright\rangle}
\newcommand{\abs}[1]{\mleft|#1\mright|}
\newcommand{\size}[1]{\#\mleft(#1\mright)}

\newcommand{\restrict}[1]{\left.#1\right|}

\newcommand{\vol}[2][]{\ifthenelse{ \equal{#1} {} }
    {\mathrm{Vol}\parens{#2}}
    {\mathrm{Vol}_{#1}\parens{#2}}
}
\newcommand{\Mahlervol}[1]{\mathcal{P}\parens{#1}}
\newcommand{\mahlervol}[2][]{\vol[#1]{#2}\cdot\vol[#1]{#2^\circ}}

\newcommand{\Conv}[1]{\mathrm{Conv}\parens{#1}}

\newcommand{\interior}[1]{\mathrm{int}\parens{#1}}

\newcommand{\aff}[1]{\mathrm{aff}\parens{#1}}

\newcommandx{\Poly}[3][1=,2=,3=]{
\ifthenelse{ \equal{#1} {} } {\mathcal{P}}
{
    \ifthenelse{ \equal{#2} {} } {\mathcal{P}^{#1}}
    {
        \ifthenelse{ \equal{#3} {} }
        {\mathcal{P}^{#2}_{\mathrm{#1}}}
        {\mathcal{P}^{#3}_{{#1}, \mathrm{#2}}}
        }
    }
}

\newcommand{\R}{\mathbb{R}}
\newcommand{\N}{\mathbb{N}}

\newcommand{\eps}{\varepsilon}
\newcommand{\defeq}{\vcentcolon=}

\begin{document}

\maketitle

\begin{abstract}
    In this paper, we will prove a result about the structure of polytopes whose floating bodies are not smooth.
    This will then be used to prove a result about the combinatorial structure of minimizers of the volume product.
\end{abstract}

\newpage

\section{Introduction}\label{sec:introduction}
Ever since Mahler's conjecture was first stated in 1939~\cite{Mahler1939}, the volume product $\mahlervol{K}$ has been of great interest in the field of convex geometry.
In this paper, we will use a result by Meyer and Reisner~\cite{MR2008} about the volume product to prove a new result about floating bodies.
We will then use this new result about floating bodies to prove a new result about the combinatorial structure of minimizers of the volume product.
The structure of the paper is as follows.
In this section, we state known results on the Mahler conjecture and about floating bodies,
together with the main results of this paper, \cref{thm:main-mahler} and \cref{thm:main-floating-body}.
In \cref{sec:preliminaries}, we state preliminary facts required for our main results.
The two main results of \cref{sec:preliminaries} are \cref{thm:polytope-perturbations} and \cref{cor:symmetric-single-vertex-movement-implies-affine-volume}.
Finally, we end by proving our two main results of the paper in \cref{sec:main-results}.

\subsection{Mahler's Conjecture}

Whenever we speak of a \textbf{convex body} in this paper, we mean a compact convex set $K \subset \R^n$ for some $n \ge 1$ with non-empty interior.
For any convex body $K$, its volume product (or Mahler product) is defined as
\[
    \Mahlervol{K} = \min_{z \in \interior{K}} \vol{K} \cdot \vol{\parens{K - z}^\circ}.
\]
Here, as usual, the dual of a compact convex set $K$ containing $0$ in its interior is defined as
\[
    K^\circ \defeq \Set{y \in \R^n | \forall x \in K: \triangles{x, y} \le 1}.
\]
The unique $z \in \R^n$ attaining the minimum is called the Santal\'o point of $K$.
When $K$ is symmetric\footnote{In this paper, a convex body described as symmetric is understood as being centrally symmetric, i.e. $K = -K \defeq \Set{-x | x \in K}$.}, this point is $0$,
and the definition of the volume product simplifies to $\Mahlervol{K} = \mahlervol{K}$.
The Mahler product is an affine invariant, and thus a natural question is to ask which convex bodies minimize this value, and which maximize it.
The maximum is known to be attained if and only if $K$ is an ellipsoid. This is called the Blaschke-Santal\'o inequality, see~\cite{Blaschke1917} and~\cite{Santalo1949}.
The Mahler conjecture and its symmetric variant, which are still open in general, state which bodies attain the minimum.

\begin{conjecture}[Mahler Conjecture]
    Let $K \subset \R^n$ be a convex body. Then
    \[
        \Mahlervol{K} \ge \frac{{(n + 1)}^{n + 1}}{{(n!)}^2}
    \]
    with equality only when $K$ is an $n$-dimensional simplex.
\end{conjecture}

\begin{conjecture}[Symmetric Mahler Conjecture]
    Let $K \subset \R^n$ be a symmetric convex body. Then
    \[
        \Mahlervol{K} \ge \frac{4^n}{n!}
    \]
    with equality only when $K$ is an invertible linear transformation of a Hanner polytope\footnote{
        Hanner polytopes are the polytopes generated with the operations $(P, Q) \mapsto P \times Q$ and $P \mapsto Q^\circ$,
        beginning with the line segment $[-1, 1] \subset \R$.
        See for example~\cite{Kim2014} for more details.
        For the sake of this paper, it is sufficient to recall that the cube ${[-1, 1]}^n$ is a Hanner polytope.
    }.
\end{conjecture}

In 1939, Mahler~\cite{Mahler1939} proved the 2-dimensional case of both versions of the conjecture,
additionally proving the equality case when $K$ is a polytope.
In 1986, Reisner~\cite{Reisner1986} proved the symmetric conjecture when $K$ is a zonoid,
and as a corollary proved the equality case of the symmetric conjecture in two dimensions.
In 1986 and 1987, Meyer~\cite{Meyer1986} and Reisner~\cite{Reisner1986} separately proved the symmetric conjecture on the class of unconditional convex bodies.
In 1987, Bourgain and Milman~\cite{BM1987} proved the lower bounds up to an exponential factor.
In 1991, Meyer~\cite{Meyer1991} proved the equality case of the asymmetric conjecture in two dimensions. 
In 2010, Reisner, Schütt, and Werner~\cite{RSW2010} proved that the minimizers must have a
Gaussian curvature of $0$ at every point where it is defined, suggesting that the minimizers are polytopes.
In 2011, Kim and Reisner~\cite{KR2011} proved the simplex is a local minimum of the volume product,
and in 2014, Kim~\cite{Kim2014} proved that Hanner polytopes are local minima of the volume product within the class of symmetric convex bodies.
In 2017, Iriyeh and Shibata proved the 3-dimensional symmetric conjecture~\cite{IS2017}.
See~\cite{FHMRZ2022} for a simplified version of their proof.
In May 2026, the 3-dimensional conjectures (both the symmetric and non-symmetric cases) were proved by Chen, Li, Xi, and Xu~\cite{CLXX2026}.
In this paper we will prove the following result on the combinatorial structure of the minimizers of the Mahler conjecture.

\begin{restatable}{theorem}{mainmahler}\label{thm:main-mahler}
    Let $l \in \N$, and let $K \subset \R^n$ be a symmetric polytope that is a local minimizer of the Mahler product among all symmetric polytopes in $\R^n$ with at most $l$ vertices.
    Then one of two cases occurs:
    
    \begin{enumerate}[label=\textit{(\roman*)}]
        \item Every facet of $K$ contains at least $n - 1$ vertices of degree at least $n + 1$ (degree is counted with respect to the 1-skeleton of $K$).
        \item $K$ is a linear transformation of the unit cube, $K = T\parens{{[-1, 1]}^n}$.
    \end{enumerate}
\end{restatable}

We remark that although it is unclear whether a minimizer of the Mahler conjecture must be a polytope,
among the set of all polytopes with at most $l$ vertices (for any $l$), there must be a minimizer,
and thus the above theorem is not an empty statement.
This result will come as a direct consequence of another result we will prove in this paper, about floating bodies.

\subsection{Floating Bodies}

The floating body is a construction that, given a convex body $K \subset \R^n$ and a parameter $0 < \delta < \frac{1}{2}$, produces a new convex body $K_{[\delta]} \subset K$
that is obtained from cutting off all possible ``caps'' of $K$ of volume $\delta \cdot \vol{K}$, namely
\[
    K_{[\delta]} \defeq \bigcap \Set{A \subset \R^n | \substack{\text{$A$ is a halfspace} \\ \vol{K \setminus A} = \delta \cdot \vol{K}}}.
\]


The floating body was first introduced by Dupin in 1822~\cite{Dupin1822}.
A case of interest is when all caps in the definition of $K_{[\delta]}$ also support it.
This is known to occur in the case when $K$ is smooth and $\delta$ is sufficiently small~\cite{Leichtweiss1986},
and also when $K$ is symmetric, for any $\delta$~\cite{Werner2005}.
It is known that the shape of a floating body is strictly convex~\cite{Werner2005}.
Our second result in this paper characterizes the shape of a polytope when its floating body is not smooth at some point.

\begin{restatable}{theorem}{mainfloatingbody}\label{thm:main-floating-body}
    Let $K \subset \R^n$ be a symmetric polytope, let $0 < \delta < \frac{1}{2}$, and consider the floating body $K_{[\delta]}$.
    Assume there exists a point $x \in \partial{K_{[\delta]}}$ at which $K_{[\delta]}$ admits at least two supporting hyperplanes.
    Let $H$ be a supporting hyperplane of $K_{[\delta]}$ at $x$.

    Then $K \cap \Conv{H, -H}$ is a linear transformation of a prism.
    That is, there is a symmetric polytope $L \subset \R^{n-1}$ together with an invertible linear transformation $T: \R^n \rightarrow \R^n$ such that
    \[
        K \cap \Conv{H, -H} = T(L \times [-1, 1]).
    \]
\end{restatable}

\section{Preliminary facts}\label{sec:preliminaries}
\subsection{Duality}\label{sec:duality}
This paper will use a few different notions of duality.
The first, and most common, is a notion of duality of convex bodies.
When $K \subset \R^n$ is a convex body containing $0$ in its interior, the dual body of $K$, sometimes called its ``polar body'', is defined as
\[
    K^\circ \defeq \Set{y \in \R^n | \forall x \in K: \triangles{x, y} \le 1}.
\]

The second notion of duality is a duality between affine spaces.
Whenever $L \subset \R^n$ is an affine space not containing $0$, we define the dual affine space of $L$ as
\[
    L^\circ \defeq \Set{y \in \R^n | \forall x \in L: \triangles{x, y} = 1}.
\]
When $0 \ne x \in \R^n$, we may for convenience denote by $x^\circ$ the dual of the affine $0$-dimensional space $\curlies{x}$.
For any closed convex set $K \subset \R^n$ and a closed set $L \subset \R^n$, we say that $L$ supports $K$ if it intersects $\partial{K}$ but does not intersect $\interior{K}$.
Then it follows directly from the two definitions of duality that for any convex body $K \subset \R^n$ and affine space $L \subset \R^n$ for which their duals are defined,
$L$ supports $K$ if and only if $L^\circ$ supports $K^\circ$.

While the symbols for $K^\circ$ and $L^\circ$ are the same,
it will always be clear from context which duality is being used,
as there is no set for which both duals are defined.

This paper will also consider a third notion of duality, between facets of polytopes.
Let $P \subset \R^n$ be a polytope with $0 \in \interior{P}$, and $F \subset P$ be a face.
Then the dual face of $F$ is defined as
\[
    F^\diamond \defeq \aff{F}^\circ \cap P^\circ
\]
where $\aff{F}$ is the affine span of $F$.
Then $F^\diamond$ is a face of $P^\circ$, and it satisfies $\dim{F} + \dim{F^\diamond} = n - 1$.

Note that the dual face of $F$ depends on the polytope $P$ as well as the face $F$.
That is, if there are two polytopes $P_1, P_2 \subset \R^n$ such that $F$ is a face of both,
then $F^\diamond$ may be different for $P_1$ and $P_2$.
Thus the operation ${(\cdot)}^\diamond$ is not well-defined,
but in this paper it will always be clear from context which polytope is being used.

\subsection{Perturbations of polytopes}\label{sec:perturbations}
The purpose of this section is to find sufficient conditions under which perturbing a polytope keeps its combinatorial structure.
To do this, we must define the \textit{face lattice} of a polytope.
We will recall some definitions and results, and for more details, see for example~\cite[Sections 2.2-2.3]{Zie1995}.

Let $K \subset \R^n$ be a polytope (not necessarily full-dimensional).
Any bijection $v: X \rightarrow V(K)$ is called a \textbf{labelling} of $K$.
For any $i \in X$, instead of writing $v(i)$ for the vertex, we will usually write $v_i$.
Thus,
\[
    V(K) = \Set{v_i | i \in X}
\]
For convenience we will usually have $X = [m] \defeq \curlies{1, 2, 3, \dots, m}$, and write
\[
    V(K) = (v_1, v_2, \dots, v_m).
\]
When $K$ has a labelling $v: X \rightarrow V(K)$, we say that ``$K$ is labelled over $X$'', and for any $I \subset X$, denote
\[
    K[I] = \Conv{\Set{v_i | i \in I}}.
\]
Given $K$ and $v$, we may then define:
\[
    L(K) = \Set{I \subset X | K[I] \text{ is a face of } K}
\]
where in this context, ``is a face'' includes the empty face and the full polytope.
The set $L(K)$ is a lattice, partially ordered by inclusion.
It is called the \textbf{face lattice} of $K$, and we say that $L(K)$ is labelled over $X$.
Note that the lattice $L(K)$ depends on the labelling $v$, but in this paper the labelling used will always be obvious from context.
When $K[I]$ is a face of $K$, sometimes we will associate $K[I]$ with $I$, and call $I$ a face of $K$.

The dimension of a face is encoded in the face lattice, as any maximal chain in $L(K)$ is of length $\dim{K} + 1$,
and the dimension depends only on the position of the face in the chain.
Therefore $L(K)$ encodes the dimension of $K$. Given a lattice $L$ that can be the face lattice of a polytope,
we may say that $L$ is an $n$-dimensional face lattice, writing $\dim L = n$, when $n$ is the dimension of the polytopes whose face lattice is $L$.
For any $i \in \curlies{0, 1, 2, \dots, n}$, denote by ${L(K)}^i$ the $i$-dimensional part of the lattice,
i.e.\ the subset of $L(K)$ that are $i$-dimensional faces. For example, ${L(K)}^{n-1}$ is the set of all facets of $K$.
For any $i \in \curlies{0, 1, 2, \dots, n}$, we define the $i$-skeleton of $K$ to be the union of all its $i$-dimensional faces, namely
\[
    \bigcup_{A \in {L(K)}^i} K[A].
\]
Because the dimensions of facets are encoded in the face lattice, a face lattice $L$ does not depend on a polytope $K$ to define its $i$-dimensional part,
which will be denoted by $L^i$.
Finally, for any face lattice $L$ labelled over $X$,
we denote by $\Poly[n](L)$ the set of all polytopes in $\R^n$ labelled over $X$ whose face lattice is equal to $L$.
Note that $\dim L \le n$, but equality is not necessary. If $\dim L = n$, then any $K \in \Poly[n](L)$ has nonempty interior.
Additionally, we will use the notation $\Poly(L) = \bigcup_{n \in \N} \Poly[n](L)$.

For any lattice $L$ labelled over $X$ and a subset $Y \subset X$, we define the restricted lattice
\[
    \restrict{L}_Y \defeq \Set{A \in L | A \subset Y}.
\]
When $Y \in L$, the lattice $\restrict{L}_Y$ is a face lattice,
and any polytope $K \in \Poly[n](L)$ satisfies $K[Y] \in \Poly[n](\restrict{L}_Y)$.

It is well-known that face lattices behave well with duality. Any polytope $P$ containing $0$ in its interior
induces an isomorphism $L(P^\circ) \cong {L(P)}^{op}$, where ${(\cdot)}^{op}$ is the operation of swapping the join and meet operations of the lattice.
The isomorphism behaves well with the ${(\cdot)}^\diamond$ operation, in the sense that if $A \in L(P)$ is the face that matches $B \in L(P^\circ)$ under this isomorphism,
then ${(P[A])}^\diamond = (P^\circ)[B]$.
For details, see~\cite[Corollaries 2.13-2.14]{Zie1995}.

\begin{lemma}\label{lem:subset-of-face-indices}
    Let $K$ be a polytope labelled over $X$, and let $I \subset X$.
    Then $K[I] \subset \partial{K}$ if and only if there exists $J \in L(K)$ with $I \subset J$.
\end{lemma}

\begin{proof}
    $(\Leftarrow)$
    Assume $I \subset J \in L(K)$. Then $K[I] \subset K[J]$,
    and since $J \in L(K)$, $K[J]$ is a face. Thus $K[J] \subset \partial{K}$, and therefore also $K[I] \subset \partial{K}$.

    $(\Rightarrow)$
    Assume $K[I] \subset \partial{K}$. Then $K[I]$ and $\interior{K}$ are two disjoint convex sets.
    Thus there is a hyperplane $H$ that defines two closed halfspaces $H^+$ and $H^-$ such that $K[I] \subset H^+$ and $\interior{K} \subset H^-$.
    Then $K[I] \subset H$ and $K \subset H^-$, meaning $K \cap H$ is a face of $K$.
    Denoting by $J$ the indices of $V(K)$ that belong to $H$, this gives the required statement.
\end{proof}

The main theorem of this section gives a sufficient condition for when small perturbations of a polytope keep its face lattice unchanged.

\begin{theorem}\label{thm:polytope-perturbations}
    Let $K \subset \R^n$ be a full-dimensional polytope with a labelling $v: [m] \rightarrow V(K)$.
    Then there exists $\eps > 0$ such that for any $v': [m] \rightarrow \R^n$,
    if for all $i \in [m]$,
    \[
        \|v_i - v_i'\| < \eps,
    \]
    and if additionally for all $I \in {L(K)}^{n-1}$,
    \begin{gather}\label{eqn:dimension-requirement}
        \dim \aff{\Set{v_i' | i \in I}} \le n - 1,
    \end{gather}
    then $K' = \Conv{\mathrm{Im}(v')}$ is a full-dimensional polytope whose vertices are labelled by $v'$,
    such that $L(K') = L(K)$.
\end{theorem}

\begin{proof}
    By continuity, it is possible to choose a small enough $\eps$ such that
    for every $v': [m] \rightarrow \R^n$ such that all $i \in [m]$
    satisfy $\|v_i - v_i'\| < \eps$,
    \begin{enumerate}[label=\textit{(\roman*)}]
        \item All of $\mathrm{Im}(v')$ are vertices of $K' = \Conv{\mathrm{Im}(v')}$.
        \item For every $I \subset [m]$, $\dim\aff{K'[I]} \ge \dim\aff{K[I]}$.
        \item Any $n + 1$ indices $\curlies{j_1, j_2, \dots, j_{n+1}} \subset [m]$ such that
        \[
            \det{(v_{j_2} - v_{j_1}, v_{j_3} - v_{j_1}, v_{j_4} - v_{j_1}, \dots, v_{j_{n+1}} - v_{j_1})} \ne 0
        \]
        satisfy that the value of
        \[
            \det{(v'_{j_2} - v'_{j_1}, v'_{j_3} - v'_{j_1}, v'_{j_4} - v'_{j_1}, \dots, v'_{j_{n+1}} - v'_{j_1})} \ne 0
        \]
        and the two values have the same \textbf{sign}.
        \item For every $I \subset [m]$, if $\aff{K[I]} \cap \interior{K} \ne \emptyset$, then $\aff{K'[I]} \cap \interior{K'} \ne \emptyset$.
    \end{enumerate}

    We will see that such an $\eps$ is a valid choice for the theorem.
    Let $v': [m] \rightarrow \R^n$ such that $\|v_i - v_i'\| < \eps$ for all $i \in [m]$,
    and assume that for all $I \in L(K)$ describing a facet of $K$,
    \[
        \dim \aff{\Set{v_i' | i \in I}} \le n - 1.
    \]
    Define $K' = \Conv{\mathrm{Im}(v')}$.
    By \textit{(i)}, $K'$ is a polytope whose vertices are labelled by $v'$.
    We claim
    \[
        L(K') = L(K)
    \]
    Since all faces of $L(K)$, including their dimensions, can be deduced from ${L(K)}^{n-1}$,
    it is enough to show that
    \[
        {L(K')}^{n-1} = {L(K)}^{n-1}.
    \]

    First, let $I \in {L(K)}^{n-1}$ be a facet of $K$. Then
    \[  
        n - 1 \overset{(\ref{eqn:dimension-requirement})}{\ge} \dim \aff{K'[I]} \overset{\textit{(ii)}}{\ge} \dim\aff{K[I]} = n - 1,
    \]
    and so
    \[
        \dim \aff{K'[I]} = n - 1.
    \]
    Since $K[I]$ is a facet of $K$, it must be that all of $K[[m] \setminus I]$ lie on the same open halfspace bounded by $\aff{K[I]}$.
    By \textit{(iii)}, all of $K'[[m] \setminus I]$ must lie on the same open halfspace bounded by $\aff{K'[I]}$.
    Then $\aff{K'[I]}$ must be a supporting hyperplane of $K'$, and $I \in {L(K')}^{n-1}$,
    proving that ${L(K)}^{n-1} \subset {L(K')}^{n-1}$.

    Next, let $I \in {L(K')}^{n-1}$.
    It must be that $K[I] \subset \partial{K}$,
    as otherwise this implies $K[I] \cap \interior{K} \ne \emptyset$, and by \textit{(iv)} we would have $K'[I] \cap \interior{K'} \ne \emptyset$,
    contradicting $I$ being a facet of $K'$.    
    By \cref{lem:subset-of-face-indices}, there is $J \in {L(K)}^{n-1}$ such that $I \subset J$.
    We proved ${L(K)}^{n-1} \subset {L(K')}^{n-1}$, therefore $I \subset J \in {L(K')}^{n-1}$.
    Since $I\in {L(K')}^{n-1}$, this means $I = J$, and $I \in {L(K)}^{n-1}$.
    Therefore ${L(K)}^{n-1} = {L(K')}^{n-1}$, and consequently $L(K) = L(K')$.
\end{proof}

\subsection{Triangulations of polytopes}\label{sec:triangulations}
The purpose of this section is to understand how the volume of a polytope behaves when conditioned on a given face lattice.
We begin by defining a triangulation, then citing a result that allows us to construct a triangulation
that simultaneously triangulates all polytopes with a given face lattice.
The following definition is equivalent to that in~\cite{LS2017}.

\begin{definition}
    Let $K \subset \R^n$ be a full-dimensional convex polytope labelled over $X$.
    Then a set $T \subset 2^X$ is called a \textbf{triangulation} of $K$ if
    \begin{enumerate}
        \item For each $A \in T$, $K[A]$ is a full-dimensional simplex.
        \item $K = \bigcup_{A \in T} K[A]$.
        \item Every $A, B \in T$ satisfy that $K[A] \cap K[B]$ is a (possibly empty) face of both $K[A]$ and $K[B]$.
    \end{enumerate}
\end{definition}

From the third property of the definition of a triangulation, it follows that if $A \ne B \in T$ satisfy $\vol[n-1]{K[A] \cap K[B]} > 0$,
then $\size{A \cap B} = \dim K$.

\begin{remark}\label{rem:vertices-on-opposite-sides}
    Let $T$ be a triangulation of $K$, and let $A \ne B \in T$ be simplices with a facet in common.
    Then the two vertices $K[A \setminus B]$ and $K[B \setminus A]$ must lie on opposite sides of $\aff{K[A \cap B]}$.
\end{remark}

If $T$ is a triangulation of $K$ for all $K \in \Poly(L)$, we say that $T$ is a triangulation of $L$.
We are interested in the following proposition:

\begin{proposition}\label{prop:triangulation-of-face-lattice}
    Let $L$ be an $n$-dimensional face lattice whose vertices are labelled over a set $X$.
    Then there exists a set $T \subset 2^X$ that is a triangulation of $L$.
\end{proposition}

This proposition is a direct consequence of~\cite[Lemma 6.1]{JZ2004},
which has been slightly rephrased to work with the notations in this paper:

\begin{lemma}[\text{\cite[Lemma 6.1]{JZ2004}}]
    Let $P$ be a $d$-dimensional polytope labelled by $v: [n] \rightarrow V(P)$.
    Then a set $\curlies{x_1, x_2, \dots, x_d} \subset [n]$ corresponds to a facet of the pulling triangulation of $\partial{P}$
    (with respect to the order given by the chosen vertex labelling) if and only if there is a complete flag of faces
    \[
        \emptyset \subset G_0 \subset G_1 \subset \dots \subset G_{d-1} \subset P,
    \]
    such that $v_{x_i}$ is the smallest vertex in $G_{d-i}$ for $1 \le i \le d$, that is, if there are facets
    $F_1, \dots, F_d$ of $P$ such that
    \[
        v_{x_i} = \min \parens{F_1 \cap \dots \cap F_i}
    \]
    for $1 \le i \le d$.
\end{lemma}

The meaning of ``pulling triangulation'' is defined in~\cite{LS2017}, and is not important for this paper.
The only important fact is that a pulling triangulation is a triangulation, and that the construction in~\cite[Lemma 6.1]{JZ2004}
depends only on the face lattice of the polytope, giving us \cref{prop:triangulation-of-face-lattice}.

Our next objective will be to prove \cref{cor:symmetric-single-vertex-movement-implies-affine-volume}
which states that in certain situations, the volume of a polytope under perturbations of its vertices behaves affinely.
We will achieve this by first proving the following more general theorem:

\begin{theorem}\label{thm:smooth-volume-by-face-lattice}
    Let $L$ be an $n$-dimensional face lattice labelled over $[k]$.
    Then there exists a polynomial $f: {(\R^n)}^k \rightarrow \R$
    such that every polytope $K \in \Poly[n](L)$ whose vertex labelling is $v$ satisfies
    \[
        \vol{K} = \abs{f(V(K))} \defeq \abs{f\parens{v_1, v_2, \dots, v_k}}.
    \]
    Additionally, any vectors $w_1, w_2, \dots, w_k \in \R^n$ and constants $c_1, c_2, \dots, c_k \in \R$
    satisfy that the mapping
    \[
        v \mapsto f(w_1 + c_1 v, w_2 + c_2 v, \dots, w_k + c_k v)
    \]
    is affine.
\end{theorem}

\begin{proof}
    Let $T \subset 2^{[k]}$ be a triangulation of $L$ as in \cref{prop:triangulation-of-face-lattice}.
    For each $A \in T$, order the indices in $A$ in an arbitrary order,
    \[
        A = (i_1, i_2, \dots, i_{n+1}).
    \]
    The volume of the simplex $K[A]$ is given by
    \[
        \vol{K[A]} = \abs{\frac{1}{n!} \cdot \det \begin{bmatrix} v_{i_2} - v_{i_1} \\ v_{i_3} - v_{i_1} \\ \vdots \\ v_{i_{n+1}} - v_{i_1} \end{bmatrix}}
    \]
    where the determinant is a polynomial in $v_{i_1}, v_{i_2}, \dots, v_{i_{n+1}}$.
    Define $f_A$ to be the polynomial such that
    \begin{gather*}
        f_A: \parens{\R^n}^k \rightarrow \R \\
        \vol{K[A]} = \abs{f_A(V(K))}.
    \end{gather*}
    Note that $f_A$ ignores all of its inputs except for the vertices belonging to $A$.
    
    This allows $\vol{K}$ to be expressed as follows
    \[
        \vol{K} = \sum_{A \in T} \vol{K[A]} = \sum_{A \in T} \abs{f_A(V(K))}.
    \]
    Next, we will pick signs $\sigma_A \in \curlies{-1, 1}$ such that
    for every $A, B \in T$, and every $K \in \Poly[n](L)$,
    \begin{gather}
        \mathrm{sign}(\sigma_A \cdot f_A(V(K))) = \mathrm{sign}(\sigma_B \cdot f_B(V(K))). \label{eqn:all-signs-are-the-same}
    \end{gather}
    This will make the polynomial $f = \sum_{A \in T} \sigma_A \cdot f_A$ satisfy $\vol{K} = \abs{f(V(K))}$.
    In fact, it would also satisfy that
    \[
        v \mapsto f(w_1 + c_1 v, w_2 + c_2 v, \dots, w_k + c_k v)
    \]
    is affine, since this is satisfied by each summand,
    as they are each a determinant where each row is of the form $w + c \cdot v$, for constants $w \in \R^n$ and $c \in \R$.

    It is left to pick signs $\parens{\sigma_A}_{A \in T}$ that satisfy (\ref{eqn:all-signs-are-the-same}).
    Consider the graph $G$, whose vertices are elements of $T$,
    with an edge between $A, B \in T$ when $\size{A \cap B} = n$.
    Since $T$ is a triangulation, this graph is connected.
    Let $G_0 \subset G$ be a spanning tree of $G$.
    For any $A, B \in T$ connected by an edge,
    let $\tau(A,B) \in S_{n+1}$ be the permutation that reorders the elements in $A$ so that the elements it shares with $B$ are in the same positions
    as in $B$.
    Now, choose signs $\curlies{\sigma_A}_{A \in T} \subset \curlies{-1, 1}$
    such that for any $A, B \in T$ connected by an edge in $G_0$,
    \[
        \sigma_A \cdot \sigma_B = -\mathrm{sign}(\tau(A,B)).
    \]
    Due to \cref{rem:vertices-on-opposite-sides}, (\ref{eqn:all-signs-are-the-same}) is satisfied for every $A, B \in T$ sharing an edge in $G_0$.
    Since $G_0$ spans $G$, for every $K \in \Poly[n](L)$, the values $\mathrm{sign}\parens{\sigma_A \cdot f_A(V(K))}$ are equal for all $A \in T$.
    Therefore (\ref{eqn:all-signs-are-the-same}) is satisfied in general, completing the theorem.
\end{proof}

\begin{corollary}\label{cor:symmetric-single-vertex-movement-implies-affine-volume}
    Let $L$ be an $n$-dimensional face lattice labelled over $[2k]$,
    and fix $x_2, x_3, \dots, x_k \in \R^n$.
    Denote by $U \subset \R^n$ the set of points such that for every $x \in U$,
    the polytope
    \[
        K(x) \defeq \Conv{\pm x, \pm x_2, \pm x_3, \dots, \pm x_k}
    \]
    when labelled with $v: [2k] \rightarrow \R^n$,
    \[
        (v_1, v_2, v_3, \dots, v_{2k}) = (x, -x, x_2, -x_2, x_3, -x_3, \dots, x_k, -x_k)
    \]
    satisfies $L(K(x)) = L$.
    Then the mapping $U \rightarrow \R$ defined as
    \[
        x \mapsto \vol[n]{K(x)}
    \]
    is an affine function.
\end{corollary}

\begin{proof}
    This follows immediately from the second part of \cref{thm:smooth-volume-by-face-lattice},
    by setting
    \begin{gather*}
        (c_1, c_2, \dots, c_{2k}) = (1, -1, 0, 0, \dots, 0) \\
        (w_1, w_2, \dots, w_{2k}) = (0, 0, x_2, -x_2, x_3, -x_3, \dots, x_k, -x_k).
    \end{gather*}
\end{proof}

\subsection{Singularities of floating bodies}\label{sec:singularities}
The purpose of this section is to define what we mean by a singularity of a convex body, and to state and prove \cref{lem:minimum-implies-floating-singularity},
which gives a connection between singularities and local minimizers of the Mahler product.

Let $A \subset \R^n$ be a convex body.
The \textbf{support function} of $A$ is a function $h_A: \R^n \rightarrow \R$ defined by
\begin{align*}
    h_A(u) &= \sup_{x \in A} \triangles{x, u}.
\end{align*}
The \textbf{supporting hyperplane} in a direction $u \in \R^n \setminus \curlies{0}$ is defined as
\[
    H_A(u) \defeq \Set{y \in \R^n | \triangles{y, u} = h_A(u)}.
\]
Finally, for $x \in \partial{A}$, we define the \textbf{supporting directions} of $x$
to be the set of directions of supporting hyperplanes containing $x$. Namely,
\[
    \mathrm{Supp}_A(x) \defeq \Set{u \in \R^n | x \in H_A(u)} \cup \curlies{0}.
\]
We remark that $\mathrm{Supp}_A(x)$ is a nonempty convex cone, and that it satisfies $1 \le \dim \mathrm{Supp}_A(x) \le n$.

\begin{definition}\label{def:convex-singularity}
    Let $A \subset \R^n$ be a convex body. For any point on its boundary $x \in \partial{A}$,
    let $d = \dim \mathrm{Supp}_A(x) - 1$.
    If $d \ge 1$, we say that $x$ is a \textbf{singularity} of $A$, and call $d$ the \textbf{dimension} of the singularity.
    If $d = n - 1$, we say that $x$ is a \textbf{full-dimensional} singularity.
\end{definition}

For example, in the cube ${[-1, 1]}^3$, the point $(1, 1, 1)$ is a full-dimensional singularity (of dimension $2$),
and the point $(1, 1, 0)$ is a singularity of dimension $1$.
More generally, for any $n$-dimensional polytope $K$, a point $x \in \partial{K}$ is a singularity of dimension $n - i - 1$,
whenever $i$ is the minimal dimension of a face of $K$ containing $x$, and $i < n - 1$.


For a polytope $K$ and a vertex $v \in K$, we denote by $\deg_K(v) \in \N$
the degree of the vertex $v$ with respect to the $1$-skeleton of $K$.
Denote by $F(K)$ the facets of $K$, and by $H^-(K)$ the halfspaces supporting $K$ at its facets.

For any $r > 0$ and for every halfspace $H^-$, denote $H^- + r \defeq H^- + B(r)$, the $r$-neighborhood of the halfspace.
Note that this is also a halfspace. Our next lemma gives a connection between Mahler product minimizers and floating bodies containing a singularity.

\begin{lemma}\label{lem:minimum-implies-floating-singularity}
    Let $K$ be a local minimizer of the Mahler product among $n$-dimensional symmetric polytopes with at most $2k$ facets.
    Denote
    \begin{align*}
        F(K) &= \Set{\pm F_1, \pm F_2, \dots, \pm F_k} \\
        H^-(K) &= \Set{\pm H^-_1, \pm H^-_2, \dots, \pm H^-_k}.
    \end{align*}
    (As the value of $k$ will not be used in any other way in the statement of the lemma,
    we may assume without loss of generality that $K$ has precisely $2k$ facets.)
    Let
    \[
        m = \dim \aff{\curlies{v \in V(F_1) | \deg_K(v) > n}}.
    \]
    Let $\eps > 0$, and denote
    \[
        K_\eps = (H_1^- + \eps) \cap ((-H_1^-) + \eps) \cap \bigcap_{i=2}^k (H^-_i \cap (-H^-_i)).
    \]
    Let $0 < \delta < \frac{1}{2}$ be the unique value such that $\aff{F_1}$ is a supporting hyperplane of $\parens{K_\eps}_{[\delta]}$.
    Then the barycenter of $F_1$ is a singularity of $\parens{K_\eps}_{[\delta]}$ of dimension at least $n - m - 1$.
\end{lemma}


\begin{proof}
    By~\cite[Proposition 1, (iii-iv)]{Werner2005}, $F_1 \cap \parens{K_\eps}_{[\delta]}$ is a single point
    \[
        F_1 \cap \parens{K_\eps}_{[\delta]} = \curlies{\zeta}
    \]
    and that point $\zeta$ is the barycenter of $F_1$.
    Let
    \[
        M = \aff{\curlies{v \in V(F_1) | \deg_K(v) > n}},
    \]
    so that $m = \dim M$.
    The vertices of $K^\circ$ are $V(K^\circ) = \curlies{\pm F^\diamond_1, \pm F^\diamond_2, \dots, \pm F^\diamond_k}$,
    where ${(\cdot)}^\diamond$ is the third notion of duality mentioned in \cref{sec:duality}.
    Define $W$ to be the affine space obtained by intersecting all facets of $K^\circ$ which include $F_1^\diamond$
    and at least $n$ other vertices, namely,
    \[
        W = \bigcap \Set{ \aff{K^\circ[F]} | \substack{ F \in {L(K^\circ)}^{n-1} \\ \size{F} > n \\ 1 \in F }}.
    \]
    Since the face lattices of $K$ and $K^\circ$ are opposites, any vertex $v \in V(F_1)$ corresponds to facet $v^\diamond \in F(K^\circ)$ such that $F_1^\diamond \in V(v^\diamond)$.
    Additionally, $\deg_K(v) > n$ if and only if $v$ is contained in more than $n$ facets, which is if and only if $v^\diamond$ has more than $n$ vertices.
    Therefore,
    \begin{align*}
        M^\circ =& {\aff{\curlies{v \in V(F_1) | \deg_K(v) > n}}}^\circ \\
        =& \bigcap \Set{v^\circ | \substack{v \in V(F_1) \\ \deg_K(v) > n}} \\
        \overset{F = v^\diamond}{=}& \bigcap \Set{\aff{K^\circ[F]} | \substack{F \in {L(K^\circ)}^{n-1} \\ 1 \in F \\ \deg_K(\parens{K^\circ[F]}^\diamond) > n}} \\
        =& \bigcap \Set{ \aff{K^\circ[F]} | \substack{F \in {L(K^\circ)}^{n-1} \\ 1 \in F \\ \size{F} > n}} = W
    \end{align*}
    and so $\dim W = n - \dim M = n - m$.

    Consider $K^\circ$ as the convex hull of its vertices
    \[
        K^\circ = \Conv{\pm F_1^\diamond, \pm F_2^\diamond, \dots, \pm F_k^\diamond}.
    \]
    By the definition of $W$, replacing the two vertices $\pm F_1^\diamond$ with any other pair of vertices $\pm w$ such that $w \in W$,
    satisfies requirement (\ref{eqn:dimension-requirement}) of \cref{thm:polytope-perturbations}.
    This is because any facet that is an $(n-1)$-dimensional simplex will always remain at most $(n-1)$-dimensional when changing the locations of its vertices,
    while every other facet remains $(n-1)$-dimensional due to the definition of $W$.
    Therefore there is $U \subset W$ with $F_1^\diamond \in U$ such that $U$ is open relative to $W$, and for any $u \in U$,
    \[
        L\parens{\Conv{\parens{V(K^\circ) \setminus \curlies{\pm F_1^\diamond}} \cup \curlies{\pm u}}} = L(K^\circ).
    \]

    For any $u \in U$, denote
    \begin{align*}
        K^\circ_u &\defeq \Conv{\parens{V(K^\circ) \setminus \curlies{\pm F_1^\diamond}} \cup \curlies{\pm u}} \\
        K_u &\defeq {(K^\circ_u)}^\circ = \bigcap_{i=2}^k (H^-_i \cap (-H^-_i)) \cap {(u^\circ)}^- \cap {(-u^\circ)}^-
    \end{align*}
    where ${(u^\circ)}^-$ is the halfspace on the side of $u^\circ$ containing $0$.
    Now,
    \[
        L(K_u) = L({(K^\circ_u)}^\circ) = {L(K^\circ_u)}^{op} = {L(K^\circ)}^{op} = {({L(K)}^{op})}^{op} = L(K),
    \]
    therefore $K_u$ and $K$ share the same face lattice.

    By \cref{cor:symmetric-single-vertex-movement-implies-affine-volume}, the mapping $U \rightarrow \R$, defined as
    \[
        u \mapsto \vol{K^\circ_u}
    \]
    is an affine function, and so there is an affine space $N$ with $F_1^\diamond \in N \subset W$
    of dimension $\dim N = \dim W - 1$
    such that for all  $u \in U \cap N$,
    \begin{gather}
        \vol{K^\circ_u} = \vol{K^\circ}. \label{eqn:dual-of-constant-volume}
    \end{gather}
    The dimension of $N$ can be expressed as
    \[
        \dim N = \dim W - 1 = n - m - 1.
    \]
    Since $K$ is a local minimizer of $\mahlervol{K}$ (in the family of $n$-dimensional symmetric polytopes with at most $2k$ facets),
    there is a neighborhood $F_1^\diamond \in U' \subset U$, such that for every $u \in U'$,
    \[
        \mahlervol{K} \leq \mahlervol{K_u}.
    \]
    Together with (\ref{eqn:dual-of-constant-volume}), this implies that for any $u \in U' \cap N$,
    \begin{gather}
        \vol{K} \leq \vol{K_u}. \label{eqn:volume-of-polytope-is-minimal}
    \end{gather}
    There is $U'' \subset U$ with $F_1^\diamond \in U''$
    such that for $u \in U''$, the facet $\partial{K_\eps} \cap (H_1^- + \eps)$ of $K_\eps$ is contained in ${(u^\circ)}^+$.
    Thus for all $u \in U''$,
    \[
        K_u = K_\eps \cap \parens{u^\circ}^-.
    \]
    Denote $S = N \cap U''$. Then $S$ is a relative neighborhood of $F_1^\diamond$ in $N$.
    Due to (\ref{eqn:volume-of-polytope-is-minimal}), for any $u \in S$,
    \begin{gather}
        \vol{K} \leq \vol{K_\eps \cap \parens{u^\circ}^-} \label{eqn:floating-body-strictness}
    \end{gather}
    with equality being achieved at $u = F_1^\diamond$. Note that this might not be the unique point of equality.

    \begin{claim*}
        Every $u \in S$ satisfies $\zeta \in u^\circ$.
    \end{claim*}

    \begin{claimproof}
        Otherwise, let $u_0 \in S$ such that $\zeta \notin u_0^\circ$.
        Then either
        \begin{enumerate}
            \item $\zeta \in \interior{\parens{u_0^\circ}^-}$ or
            \item $\zeta \in \interior{\parens{u_0^\circ}^+}$.
        \end{enumerate}
        In the first case, since $\zeta \in {(K_\eps)}_{[\delta]}$,
        \[
            ({u_0}^\circ) \cap \interior{{(K_\eps)}_{[\delta]}} \ne \emptyset.
        \]
        By the definition of ${(K_\eps)}_{[\delta]}$, this means $\vol{K_{u_0}} < \vol{K}$, contradicting (\ref{eqn:floating-body-strictness}).
        In the second case, consider the line segment $[F_1^\diamond, u_0]$.
        When we move $u$ from $u_0$ to $F_1^\diamond$ along this line segment,
        the halfspace $\parens{u^\circ}^+$ initially contains $\zeta$, and rotates linearly, until at the end of the segment, $\zeta$ is at its boundary.
        Thus, for $u_1$ very close to $F_1^\diamond$ on the other side of the line segment,
        \[
            \zeta \in \interior{\parens{u_1^\circ}^-}.
        \]
        Since $S$ is a relative neighborhood of $F_1^\diamond$ in $N$, there is such a $u_1 \in S$,
        and by setting $u_0$ to be $u_1$ we may assume without loss of generality that $u_0$ belongs to the first case, which has already been proved.
    \end{claimproof}

    Finally, we claim that for every $u \in S$, the hyperplane $u^\circ$ supports $\parens{K_\eps}_{[\delta]}$.
    Otherwise, since $\zeta \in u^\circ \cap \parens{K_\eps}_{[\delta]}$,
    \[
        ({u}^\circ) \cap \interior{{(K_\eps)}_{[\delta]}} \ne \emptyset,
    \]
    which again by the definition of ${(K_\eps)}_{[\delta]}$ means $\vol{K_u} < \vol{K}$, contradicting (\ref{eqn:floating-body-strictness}).
    
    The dimension of $S$ is $\dim N = n - m - 1$, and for every $u \in S$, $u^\circ$ is a supporting hyperplane of $\parens{K_\eps}_{[\delta]}$ at $\zeta$,
    making $\zeta$ a singularity of $\parens{K_\eps}_{[\delta]}$ of dimension at least $n - m - 1$, as required.
\end{proof}

\section{Proofs of Theorems~\ref{thm:main-mahler} and~\ref{thm:main-floating-body}}\label{sec:main-results}
A key ingredient in the proof of \cref{thm:main-floating-body} will be~\cite[Proposition 7]{MR2008}.
This is a statement on shadow systems. For literature on the topic, see for example~\cite[Section 10.4]{Schneider2013}.
The proposition is as follows:

\begin{proposition}[\text{\cite[Proposition 7]{MR2008}}]\label{prop:shadow-system-equality-case}
    Let $K_t$, $t \in [a, b]$, be a non-degenerate shadow system in $\R^d$.
    Then the following are equivalent:
    \begin{enumerate}
        \item $t \mapsto \vol{K_t}$ and $t \mapsto \vol{K_t^\star}^{-1}$ are both affine functions of $t$.
        \item There exists a linear representation $\R^d \cong \R^{d-1} \times \R$,
            real numbers $v, u \in \R$, and a vector $V \in \R^{d-1}$ such that the affine transformations $A_t: \R^d \rightarrow \R^d$ defined by
            \[
                A_t(X, x) = (X, x + \parens{t - \frac{a + b}{2}}\parens{vx + \triangles{V, X} + u})
            \]
            satisfy $A_t(K_{(a + b)/2}) = K_t$ for every $t \in [a, b]$.
    \end{enumerate}
\end{proposition}

The operation $K^\star$ is defined as ${(K-s(K))}^\circ$, where $s(K)$ is the Santaló point of $K$.
We will only be using this statement when $K$ is symmetric, in which case $s(K) = 0$ and $K^\star = K^\circ$.
With this, we can prove our main results, \cref{thm:main-floating-body} and \cref{thm:main-mahler}.

\begin{proof}[Proof of \cref{thm:main-floating-body}]
    As in the statement of the theorem, let $K \subset \R^n$ be a symmetric polytope, let $0 < \delta < \frac{1}{2}$,
    let $K_{[\delta]}$ be its floating body,
    and let $x \in \partial{K_{[\delta]}}$ be a point at which $K_{[\delta]}$ admits at least two different supporting hyperplanes, $H \ne H_1$.
    Define $\gamma: [0, 1] \to \R^n$ by
    \[
        \gamma(t) = (1 - t) H^\circ + t H_1^\circ.
    \]
    Then $\gamma$ is a linear function onto the line segment $[H^\circ, H_1^\circ]$.

    For $t \in [0, 1]$, consider the shadow system
    \[
        P_t = \Conv{K^\circ \cup \curlies{\pm \gamma(t)}}
    \]
    whose dual bodies are given by
    \[
        P_t^\circ = K \cap \parens{{\gamma(t)}^\circ}^- \cap \parens{{-\gamma(t)}^\circ}^-.
    \]

    Each $P_t^\circ$ is the intersection of $K$ with two halfspaces.
    At points $t=0$ and $t=1$, these halfspaces are $\pm H$ and $\pm H_1$ respectively,
    and thus for every $t$, the boundaries of these halfspaces are supporting hyperplanes of $K_{[\delta]}$.
    Thus, the volume of $P_t^\circ$ is constant in $t$.

    As $t$ varies, so does $L(P_t)$. There is a finite number of options for $L(P_t)$.
    Denote these options by $L_1, L_2, \dots, L_m$.
    By \cref{cor:symmetric-single-vertex-movement-implies-affine-volume},
    for each $L_i$ there is an affine function $A_i: [0, 1] \to \R$ such that $A_i(t) = \vol{P_t}$ whenever $L(P_t) = L_i$.
    Since $\vol{P_t}$ is continuous, and always equal to one of $A_1(t), A_2(t), \dots, A_m(t)$,
    there is $r \in \lparen0, 1\rbrack$ such that $\vol{P_t}$ is affine on $[0, r]$.
    On the restriction of the shadow system $P_t$ from $[0, 1]$ to $[0, r]$, we can apply \cref{prop:shadow-system-equality-case}
    to get that $P_t$ is an affine image of $P_{r/2}$ for every $t \in [0, r]$.

    Other than $\pm \gamma(t)$, the vertices of $P_t$ do not change.
    As $t$ varies, all bodies $P_t$ are affine images of each other.
    Since they are polytopes, the affine transformations must map their vertices to each other.
    Therefore by continuity, these affine transformations fix all vertices of $P_t$ other than $\pm \gamma(t)$,
    and in particular, these fixed vertices must lie on a single hyperplane $N$.

    Therefore for all $t \in [0, r]$,
    all the facets of $P_t^\circ$ other than $\parens{{\gamma(t)}^\circ}$ and $\parens{{-\gamma(t)}^\circ}$ are parallel to $N^\circ$.

    In particular, this is true at $t = 0$, where $P_0^\circ = K \cap \Conv{H, -H}$ and ${\gamma(t)}^\circ = H$.
    Therefore all the facets of $K \cap \Conv{H, -H}$ other than $H$ and $-H$ are parallel to $N^\circ$.

    This means that $K \cap \Conv{H, -H}$ is a skewed prism with $\pm H$ as its bases, and the direction of the prism is $N^\circ$, as required.
\end{proof}

Now that we have proved \cref{thm:main-floating-body}, we are ready to prove \cref{thm:main-mahler}.

\begin{proof}[Proof of \cref{thm:main-mahler}]
    The proof will be by induction on $n$. When $n = 1$, the theorem is trivial, as the only polytopes in $\R$ are line segments.

    Let $n > 1$ and assume the theorem holds for $n - 1$.

    Let $K$ be a local minimizer as in the statement of the theorem.
    Assume there is a facet $F \in F(K)$ that contains at most $n - 2$ vertices of degree at least $n + 1$.
    
    Let $\eps > 0$. As in \cref{lem:minimum-implies-floating-singularity}, define $K_\eps$, $\delta > 0$, and $m$. Then
    \[
        m = \dim \aff{\curlies{v \in V(F) | \deg(v) > n}} \le (n - 3) + 1 = n - 2
    \]
    where the inequality is true by our assumption that $F$ contains at most $n - 2$ vertices of degree at least $n + 1$,
    and the bound $\dim \aff{X} \le \size{X} - 1$ for any finite set $X$.

    Then by \cref{lem:minimum-implies-floating-singularity}, $F \cap \parens{K_\eps}_{[\delta]}$ is a single point, which is a singularity of $\parens{K_\eps}_{[\delta]}$
    of dimension at least $n - m - 1 \ge n - (n - 2) - 1 = 1$.

    Since the dimension of the singularity is at least $1$, by \cref{thm:main-floating-body}
    the polytope $K_\eps \cap H_1^- \cap {(-H_1)}^- = K$ is a skewed prism with $\pm F_1$ as its bases.

    Therefore there exists an $(n-1)$-dimensional symmetric polytope $L$ and a linear transformation $T: \R^n \rightarrow \R^n$ such that
    \[
        K = T(L \times [-1, 1]).
    \]

    Since $K$ is assumed to be a local minimizer for the Mahler product in the family of $n$-dimensional symmetric polytopes with at most $l$ vertices,
    $L$ must be a local minimizer for the Mahler product in the family of $(n-1)$-dimensional symmetric polytopes with at most $l/2$ vertices.
    Otherwise, there would be a polytope $L'$ arbitrarily close to $L$ with at most $l/2$ vertices and with smaller Mahler product than $L$,
    and then $T(L' \times [-1, 1])$ would be arbitrarily close to $K$, with at most $l$ vertices and smaller Mahler product than $K$,
    contradicting the assumption that $K$ is a local minimizer.

    Every $v \in V(L)$ corresponds to a vertex $w = T((v, 1)) \in V(F) \subset V(K)$.
    They satisfy $\deg_L(v) = \deg_K(w) - 1$.
    There are at most $n - 2$ vertices $w \in V(F)$ such that $\deg_K(w) \ge n + 1$,
    and so there are at most $n - 2$ vertices $v \in V(L)$ such that $\deg_L(v) \ge n$.
    In particular, $L$ has a facet $L_0 \in F(L)$ with at most $n - 3$ vertices $v \in V(L_0)$ such that $\deg_L(v) \ge n$.
    Such an $L_0$ can be chosen by picking a vertex $v_0 \in V(L)$ with $\deg_L(v_0) \ge n$
    (if there is no such vertex, any facet can be chosen as $L_0$),
    and then letting $L_0$ be any facet that does \textbf{not} contain $v_0$
    (at least one such facet exists).
    By the induction hypothesis, $L$ is a linear transformation of the unit cube in $\R^{n-1}$,
    and therefore $K$ is a linear transformation of the unit cube in $\R^n$, as required.
\end{proof}

Note that it is not important that the minimizer in \cref{thm:main-mahler} be defined as a local minimizer among polytopes with at most $l$ \textit{vertices}.
We could instead have it be a minimizer among polytopes with at most $l$ \textit{facets}, and the proof would be identical.
The only difference in the proof would be stating that $L$ is a local minimizer among $(n-1)$-dimensional symmetric polytopes with at most $l - 2$ facets, instead of $l / 2$ vertices.
Thus \cref{thm:main-mahler} can be combined with its dual statement.

\begin{theorem}
    Let $l \in \N$, and let $K \subset \R^n$ be a symmetric polytope that is a local minimizer of the Mahler product among all symmetric polytopes in $\R^n$ with at most $l$ vertices.
    Then one of three cases occurs:
    
    \begin{enumerate}[label=\textit{(\roman*)}]
        \item Every vertex of $K$ is contained in at least $n - 1$ facets that are not simplices,
        and every facet of $K$ contains at least $n - 1$ vertices of degree at least $n + 1$.
        \item $K$ is a linear transformation of the unit cube, $K = T\parens{{[-1, 1]}^n}$.
        \item $K$ is a linear transformation of the cross-polytope, $K = T\parens{\parens{{[-1, 1]}^n}^\circ}$.
    \end{enumerate}
\end{theorem}

\bibliographystyle{plain}
\bibliography{ref}

\end{document}